\documentclass[10pt]{article}

\usepackage {amsfonts}
\usepackage {amssymb}
\usepackage {amsmath}
\usepackage{hyperref}
\usepackage{graphicx}
\usepackage{amsthm}
\usepackage {color}
\usepackage {verbatim}

\newtheorem{T}{Theorem}[section]
\newtheorem{Lm}{Lemma}[section]
\newtheorem{D}{Definition}[section]

\newcommand{\un}{\underline}

\newcommand{\A}{\forall}

\newcommand{\inr}{{\rm int}}
\newcommand{\vol}{{\rm vol}}

\newcommand{\cov}{{\rm cov}}

\newcommand{\RgC}{\mathfrak{C}}
\newcommand{\RgM}{\mathfrak{M}}

\begin{document}

\title{Rearrangements of Gaussian fields}
\author{Rapha\"el Lachi\`eze-Rey\thanks{raphael.lachieze-rey@math.univ-lille1.fr}, Youri Davydov\thanks{youri.davydov@math.univ-lille1.fr}\\USTL, Lille}

\maketitle

\paragraph{abstract}

The monotone rearrangement of a  function is the  non-decreasing function with the same distribution. The convex rearrangement of a smooth function is obtained by integrating the monotone rearrangement of its  derivative. This operator can be applied to regularizations of a stochastic process to measure quantities of interest in econometrics.

A multivariate generalization of these operators  is proposed, and  the almost sure convergence of rearrangements of regularized  Gaussian fields is given. For the Fractional Brownian field or the Brownian sheet approximated on a simplicial grid, it appears that the limit object depends on the orientation of the simplices.

\paragraph{keyword} 
random fields, rearrangement, limit theorems, random measures  

\paragraph{MSC} 60G60, 60B12, 60G57

\section*{Introduction and notation}
\label{SecInt}

The following notations will be useful. In $\mathbb{R}^d$, denote by  $+$  the Minkowski addition of sets. The operators $\rm{vol, diam, cl, int},\partial$ resp. stand for the volume, diameter, closure, interior and boundary of a Borel set. Let $\|z\|$ be the euclidean norm of a vector, and $\|z\|_{1}=\sum_{i}|z_{i}|$ its $\mathcal{L}^1$ norm, where the $z_{i}$ are the coordinates of $z$ in the canonical basis $\mathbf{e}=(\mathbf{e}_{1},\dots,\mathbf{e}_{d})$. Denote by $\lambda_{d}$ the Lebesgue measure in $\mathbb{R}^d$, and $\gamma_{d}$ the standard normal distribution. The cardinality of a finite set $E$ is denoted by $|E|$. Given two  random vector-valued variables $X, Y$, let $\rm{cov}(X,Y)$ be their covariance matrix in a predefinite basis $\mathbf{u}$, i.e. 
\begin{equation*}
\cov(X,Y)_{i,j}=\mathbb{E}X_{i}Y_{j}-\mathbb{E}X_{i}\mathbb{E}Y_{j},
\end{equation*} where the $X_{i}$ and $Y_{j}$ are the components of $X$ and $Y$ in $\mathbf{u}$. The covariance matrix of a vector is simply denoted $\cov(X,X)=\cov(X)$. The weak convergence of measures is denoted by $\Rightarrow$.

\paragraph{Preliminary example}
Consider a finite population, arbitrarily labelled with numbers $k$ in $\{1,\dots ,N\}$, for $N\in \mathbb{N}^*$. For $1\leq k \leq N$, the  member $k$ receives an income of a certain resource, denoted by a real number $g(k)$. Now let $\sigma$ be a permutation of $\{1,\dots ,N\}$ that makes the function $k\to g(\sigma(k))$ non-decreasing. Call $\tilde{g}=g\circ \sigma$ the $\emph{monotone rearrangement}$ of $g$. 

 Define $\psi(k)=\sum_{i=1}^k \tilde{g}(i),\, 1 \leq k \leq n$. Since $\tilde{g}$ is monotone, $\psi$ is convex. For $1\leq k \leq N$, $\psi(k)$ represents the total amount of resources detained by the $\frac{k}{n}$-th poorest fraction of the population. Now, call $\overline{\psi}(k)=\frac{k}{n}\psi(n)$. It is the ``equality function'', in the sense that $\psi=\overline{\psi}$ iff all incomes are equal. Also, for some distance $\delta$, the distance $\delta(\psi,\overline{\psi})$ between $\psi$ and its equality function measures the inequalities among the population. 

 If one defines $f(k)=\sum_{i=1}^k g(i),~1\leq k \leq N$, the cumulative income, $\psi$ is called the \emph{convex rearrangement} of $f$. It is indeed the only convex function which has the same increments (but in a different order), and coincides with $f$ at $N$. Consider for instance the case where $\delta$ is the $\mathcal{L}^1$ norm on $\mathbb{R}^N$, normalized by $N$. For a given cumulative income function $f$, the quantity
\begin{displaymath}
N^{-1}\|\psi-\overline{\psi}\|_{1}=\frac{1}{N}\sum_{k=1}^N \left|\psi(k)-\frac{k}{n}f(n)\right|
\end{displaymath}
     retrieves the Gini coefficient, which has played a central role in measuring economic
      inequality since its introduction by Corrado Gini at the beginning of the 20th
      century. The use of the convex rearrangement for measuring economic inequality is discussed in \cite{Z}.\\

The notion of rearrangement, defined above for a discrete population, can be generalized in the continuous framework. If $g_{1}$ is an integrable function on $[0,1]$, and $\sigma$ is a transformation of $[0,1]$ which preserves Lebesgue measure, the function defined by 
\begin{eqnarray}
\label{eq:Rrg}
g_{2}=g_{1}\circ \sigma
\end{eqnarray} 
is a rearrangement of $g$. For any function $g$, denote by $\mu_{g}$ the image of Lebesgue measure under $g$. Relation (\ref{eq:Rrg}) also implies
\begin{eqnarray}
\label{eq:RrgMsr}
\mu_{g_{2}}=\mu_{g_{1}}.
\end{eqnarray}
A function $g_{2}$ is said to be a \emph{rearrangement} of $g_{1}$ if it satisfies (\ref{eq:RrgMsr}). Remark that in general this is not equivalent to (\ref{eq:Rrg}). 
A monotone rearrangement of an integrable function $g$ on $[0,1]$  is a non-decreasing function that is a rearrangement of $g$, and is denoted by $\mathfrak{M}g$.
It is easy to see that 
every integrable function on $[0,1]$ admits a monotone rearrangement, unique up to a negligible set ( see for instance \cite{R}).

 Like in the preliminary example, a convex rearrangement of a differentiable function $f$ is a convex function $\psi$ which derivative is obtained as the rearrangement of the derivative of $f$, i.e. $\mu_{\psi'}=\mu_{ f'}$. If furthermore $\psi$ and $f$ coincide in a predetermined point $z_{0}$, then $\psi$ is \emph{the convex rearrangement of $f$}, and is denoted by $\psi=\mathfrak{C} f$.

 If a function $f$ is irregular, one can take regularizations $f_{n},n\geq 1$, and study asymptotically their rearrangements $b_{n}^{-1}\mathfrak{C}f_{n}$, under the proper renormalization $b_{n}>0$.  Also, the asymptotic rearrangement is consistent, i.e. if $\mu_{g_{n}}$ has a weak limit $\mu$, for a sequence of functions $\{g_{n};~n\in \mathbb{N}\}$, then the monotone rearrangements of the $g_{n}$ also converge, to a function $g$ satisfying $\mu_{g}=\mu$. The result is similar for convex rearrangements, i.e. the convergence of the $\mu_{f_{n}'}$ yields the convergence of the $\mathfrak{C}f_{n}'$.  
 It is of practical and theoretical interest to investigate asymptotic properties of rearrangements. It can be used, for example, to construct estimators of parameters of stochastic processes, and for measuring their fluctuations, see \cite{DZ}. There are also connections between convex rearrangement and other areas of research such as Finance Mathematics and Economics. The Lorenz curve, important in finance mathematics, is a common object in convex rearrangement of Gaussian processes. In the field of econometrics, convex rearrangement can be used to measure the indices of fluctuations of stochastic processes, related to indices of economic inequality, like the Gini index in the preliminary example, see \cite{Z}. The monotone rearrangement of a function $g$ also has a physical meaning, as the solution of the optimal transport problem with transfer plan $g$. the asymptotic convex rearrangement has been studied for many one-dimensional processes, see \cite{DZ} for a survey.\\

We propose the following generalization to a compact $K$ of $\mathbb{R}^d$. For a function $g$ integrable on $K$, call $\mu_{g}$ the image of Lebesgue measure $\lambda_{d}$ under $g$. Then a function $g_{2}$ is a rearrangement of an other function $g_{1}$ if and only if it satisfies (\ref{eq:RrgMsr}). The rearrangement is furthermore said to be monotone if $g_{2}$ is a \emph{monotone function}, i.e the gradient of a convex function. Correspondingly, a function $\psi$ is a convex rearrangement of a real function $f$ if it is convex and yields the same gradient distribution than $f$. 

In Section \ref{SecCvx}, a reduced version of the problem of optimal transport is introduced.  Brenier's theorem, originally designed for this optimal transport problem, is given, and this allows us to  rigorously define monotone and convex rearrangements in higher dimensions. We also prove that, like in the one-dimensional case, the convergence of the convex rearrangements of a family of functions $f_{n}$ is equivalent to the weak convergence of the measures $\mu_{\nabla f_{n}}$. This result, which serves later for rearranging Gaussian fields, is called  the \emph{consistency theorem}.

In section \ref{SecRdmFld}, we introduce the probabilistic framework of this paper. It consists of a  random field $X$ approximated by polygonal fields $X_{n},\,n\geq 1$, interpolating $X$ on a simplicial grid. We give in the Gaussian framework the almost sure weak convergence of the sequence of measures $\mu_{b_{n}\nabla X_{n}}=\lambda_{d}(b_{n}^{-1}\nabla X_{n})^{-1}$ to a measure $\mu$ for proper $b_{n}>0$, under weak  assumptions  on the covariance function of the field $X$. This yields according to the \emph{consistency theorem} the convergence of $b_{n}^{-1}\mathfrak{C}X_{n}$.  The almost sure convergence towards $\mu$ ensures that quantities of interest can be computed from each sample path $X_{n}$. Thus this deterministic limit object, new in the literature, can serve for estimating several quantities related to the regularity and the isotropy of $X$, or more generally to its covariance function, with only one realization.

We show in Section \ref{sec:ex} that this result applies  to the Fractional Brownian field and to the Brownian sheet, and compute the limit measure $\mu$.
At the contrary of the one-dimensional case, we observe through these examples that $\mu$ depends on the method of approximation, and in particular on the orientation of the simplices used in the triangulation. We represented on Figure \ref{fig:brw-sheet} the asymptotic convex rearrangement of the Brownian sheet on $[0,1]^2$, approximated by polygonal fields on a natural triangulation of the plane. 
 
 \begin{figure} [h]
\begin{center}
\includegraphics[scale=0.65, angle=0] {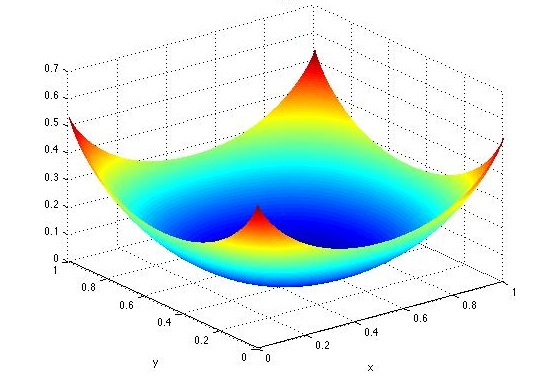}
\end{center}
\caption{Asymptotic convex rearrangement of the Brownian sheet.}
 \label{fig:brw-sheet}
\end{figure}

\section{Monotone rearrangements and optimal transport}
\label{SecCvx}
This section exposes the theoretical material required for   rearranging  multivariate functions with compact support. It is related to the optimal transport problem, in that the monotone rearrangement $\mathfrak{M}S$ of a transport plan $S$ coincides with the optimal solution to the corresponding transport problem. Then, we study the consistency of the monotone rearrangement, needed for rearranging irregular functions, the same way it is done for Brownian motion just below.\\
 
 \subsection{One dimensional case. Convex rearrangement of the Brownian motion}
 \label{sec:Int}
As has been said in the introduction, the monotone rearrangement of a function $g$ is the monotone function that yields the same distribution.
We emphasize here that the central object of the monotone rearrangement  is $\mu_{g}$, the image of Lebesgue measure under $g$. In other words, two functions have the 
same rearrangement if they have the same distribution.

  If now $f$ is an absolutely continuous function on $[0,1]$, i.e such that for almost all $x$ in $[0,1],$ $f(x)-f(0)=\int_{0}^x g(t)dt$ for some integrable function $g$, the \emph{convex rearrangement} of $f$ is the unique convex function $\psi$ verifying $\psi(0)=f(0)$ and $\psi'=\mathfrak{M}f '$ a.e.. Write $\psi=\mathfrak{C} f$, where $\mathfrak{C}$ is the \emph{convex rearrangement operator}.

 For $f$ irregular, one chooses smooth approximations $\{f_n~;~n\geq 1\}$, and studies asymptotically the rearrangements.  If there exists a sequence $\{b_n; n \geq 1\}$ and a convex function $\psi$ such that $\frac{1}{b_n}\mathfrak{C} f_n\to \psi$ a.e., $\psi$ is said to be an \emph{asymptotic convex rearrangement} of $f$ with \emph{renormalizing sequence} $\{b_n; n \geq 1\}$.

Although a rigourous study is not trivial, it is possible to understand better the convex rearrangement machinery in the case of the Wiener process. Take   $X$ a standard Brownian motion on $[0,1]$, with $X_{n}$ its piece-wise linear interpolation on $\left\{\frac{k}{n};~0\leq k \leq n\right\}$, normalized by $\sqrt{n}$ to avoid the divergence of the increments. For each $n$, $X_{n}$ is differentiable a.e., and the image of Lebesgue measure $\lambda_{1}$ under the renormalized derivative is written $$\mu_{n}=\lambda_{1} \left(\frac{1}{\sqrt{n}}X_{n}'\right)^{-1}.$$ The independence of increments implies that $\mu_{n}$ is the empirical distribution of $n$ independent normal variables, and it is clear that it converges weakly to the normal distribution $\gamma_{1}$. It is rigorously proven later, in Theorem \ref{ThmCty}, why this implies that the asymptotic convex rearrangement of $X$ on $]0,1[$ is the Lorenz curve $GL_{1}$, defined as the unique convex function with gradient distribution $\gamma_{1}$. Davydov and Vershik \cite{DV} obtained the strongest result, namely the uniform convergence of $\|\frac{1}{\sqrt{n}}\mathfrak{C}X_{n}-GL_{1}\|_\infty$ to $0$ with probability $1$.
  
 A lot of similar results are obtained with processes that have stationary increments, or are stable, see the survey \cite{DZ}. Azais and Wschebor \cite{AW} also showed that, for $X$ in a certain class of Gaussian processes, if instead of a piece-wise linear approximation, one chooses for $X_n$ a regularization of $X$ by a convolution kernel, then $X$ admits the same asymptotic convex rearrangement, namely the generalized Lorenz curve $GL_{1}$. In this case, the asymptotic convex rearrangement of $f$ seems unambiguous, up to the multiplication by a non-zero constant, in the sense that it does not depend on the approximation method. We will see in Section \ref{sec:ex} that it is not the case for anisotropic multivariate random fields. \\

\subsection{The optimal transport problem and rearrangement operators}
The problem described below is a simplified version of the traditional optimal transport problem, which is fully described and exhaustively discussed in \cite{V}.

  A company has a capacity of production per unit time represented by a measure $\mu$ on $\mathbb{R}^d$, the \emph{production measure}. The quantity produced in area $dx$ per unit time is $\mu(dx)$. This company has to deliver its production to a domain $K$ of $\mathbb{R}^d$, compact and convex, where the demand is uniformly distributed. The cost of transport between a site of production $s$ and a point $z$ in $K$ is denoted by $c(s,z)$, where the \emph{cost function} $c$ is supposed to be measurable and non-negative. A transport plan $S$ is a function which associates to each $z$ in $K$ the corresponding production site $S(z)$, where the product delivered to $z$ comes from. Let $\mu_{S}$ be the image of Lebesgue measure under $S$. We need to have, for all Borel set $B\in \mathcal{B}_{d}$,
\begin{eqnarray}
\label{eq:MasCsv}
\mu_{S}(B)=\mu(B),
\end{eqnarray}
 so that the quantity produced at each production site corresponds to the quantity of product conveyed to the distribution area. The total cost of this transport plan  is hence
\begin{eqnarray*}
C(S)=\int_{z\in K}c(z,S(z)) dz.
\end{eqnarray*}
 Assume that the cost is quadratic, i.e. $c(z,s)=\|z-\zeta\|^2$.
The optimal transport problem consists in finding a transport plan $S:K\to \mathbb{R}^d$ minimizing the cost $C(S)$ under requirement (\ref{eq:MasCsv}). Addressing this issue, suppose that a given transport plan $S$ is modified by switching  the destinations $z$ and $\zeta$ for two productions sites $S(z)$ and $S(\zeta)$ for an infinitesimal quantity of product. The new transport plan is denoted $\tilde{S}$ and the corresponding cost variation is 
$$C(\tilde{S})-C(S)=2\langle z-\zeta,S(z)-S(\zeta)\rangle(dz+d\zeta). $$ Informally, a transport plan will be in some sense locally optimal if, for all $z,\zeta\in K$,
\begin{eqnarray}
\label{eq:MonFct}
\langle z-\zeta,S(z)-S(\zeta)\rangle\geq 0.
\end{eqnarray}
 It turns out that (\ref{eq:MonFct}) and (\ref{eq:MasCsv}) indeed characterize optimal transport plans (see \cite{V}).
 
 The question that naturally arises now is about the existence of such an optimal transport plan.
 That is the purpose of the following theorem.  
\begin{T}
[Brenier]
\label{ThmBre}
~\\
Call $\mathcal{K}(K)$ the class of convex functions on $K$. Let $\mathcal{G}(K)$ be the set of monotone functions on $K$, defined by\\
\begin{eqnarray*}
\mathcal{G}(K)=\left\{\nabla \psi;~\psi \in \mathcal{K}(K), \int_{K}\|\nabla\psi\|<\infty\right\}.
\end{eqnarray*}
Then, if   $\mu$ is a measure on $\mathbb{R}^d$ with finite first moment, there is a unique monotone function in $\mathcal{G}(K)$, denoted by $\mathbf{M}_{\mu}$, such that $\lambda_d \mathbf{M}_{\mu}^{-1}=\mu$.  \end{T}
Comments, proof, and a more general result can be found  in  \cite{Bre}.
If a point $z_{0}$ of $K$ is unambiguously defined as ``starting point'', call $\mathbf{C}_{\mu}$ the convex function which gradient is $\mathbf{M}_{\mu}$, satisfying $\mathbf{C}_{\mu}(z_{0})=0$.
 The function $\mathbf{M}_{\mu}$ is the optimal solution of the transport problem with production measure $\mu$.

Theorem \ref{ThmBre} is the proper tool to define high dimensional monotone and convex rearrangements. 
\begin{D}

For an integrable function $g$ on $K$, define $\mathfrak{M}g=\mathbf{M}_{\mu_{g}}$ its monotone rearrangement.

Let $\mathcal{S}(K)$ be the class of functions which are differentiable in a.e. point of $K$ and satisfy
\begin{equation*}
\int_{K}\|\nabla f(z)\|dz<+\infty.
\end{equation*}
 For a function $f$ in $\mathcal{S}(K)$, there exists a unique convex function, denoted by $\mathfrak{C}f$, which satisfies 

\begin{eqnarray*}
\lambda_d(\nabla \mathfrak{C}f)^{-1}&=&\lambda_d (\nabla f)^{-1},\\
\mathfrak{C}f(z_{0})&=&f(z_{0}).
\end{eqnarray*}
It is called \emph{convex rearrangement} of $f$.
\end{D}

The convex rearrangement can also be defined as $\mathfrak{C}f=\mathbf{C}_{\mu_{\nabla f}}+f(z_{0})$. Given a vector-valued function $S$ on $K$, since $\mathfrak{M}S$ is the gradient of a convex function, its restriction to each segment $[z,\zeta]\subset K$ is non-decreasing, whence it satisfies (\ref{eq:MonFct}). In this regard, Theorem \ref{ThmBre} provides with $\mathfrak{M}S$ a unique solution to the optimal transport problem with transport plan $S$. Note that \cite{Bre} also gives the existence of a measure-preserving transformation $\sigma$ of $[0,1]$ such that $\mathfrak{M}S \circ \sigma = S$, provided $\mu_{S}$ is absolutely continuous with respect to Lebesgue measure, which justifies the ``rearrangement'' terminology.\\

In dimension $1$, convex rearrangement was already defined in the literature. The class $\mathcal{S}(K)$ is exactly that of absolutely continuous functions if $K$ is a compact interval of $\mathbb{R}$. Hence, it is a generalization  of absolutely continuous functions upon which we extend operator $\mathfrak{C}$.
Note that, although it is called ``convex rearrangement'', function $\mathfrak{C}f$ is not  a rearrangement of $f$ in the sense of (\ref{eq:RrgMsr}). For instance, $f$ and $\mathfrak{C}f$ do not in general yield the same maximum. Nevertheless, visually  it corresponds in some way to piling up the increments of $f$ in another order.

 \subsection{Consistency of the rearrangement operators}
In this article we deal with irregular random fields, for which we cannot a priori obtain a convex rearrangement due to the absence of gradient. In consequence, by analogy with the $1$-dimensional case, we instead investigate asymptotically  the convex rearrangement of their regularizations. Call  \emph{ asymptotic convex rearrangement} of $f$ any convex function that is the limit of renormalized convex rearrangements of regularizations of $g$. Theorem \ref{ThmCty} allows us to obtain an asymptotic convex rearrangement of a function by studying asymptotically the gradients distributions.

 In the sequel, $K$ is a convex body of $\mathbb{R}^d$, with an arbitrary starting point $z_{0}\in K$. The following theorem will be our main tool for rearranging random fields. For a compact set $L$ and a real-valued function $f$ on $L$, 
\begin{equation*}
\|f\|^L_{\infty}=\sup_{x\in L}|f(x)|,
\end{equation*} and for a vector-valued function $g$ on $L$,  
\begin{equation*}
\|g\|_{L^1}^L=\int_{L}\|g(z)\|_{1}dz.
\end{equation*}
  
\begin{T}~\\
\label{ThmCty}
Take $\{f_{n};~n\geq 1\}$ and $f$ in $\mathcal{S}(K)$, and define $g_{n}=\nabla f_n,  ~g=\nabla f$. Then the three following statements are equivalent:
\begin{eqnarray}
\label{eq:CvgMsr}
 \mu_{g_{n}}& \Rightarrow& \mu_{  g},\\
 \label{EquCty1}
\|\mathfrak{M}g_{n}- \mathfrak{M}g \|_{L^1}^L & \to & 0,\;\text{for all compact  $L$ of ${\rm int}(K)$},\\
\label{EquCty2}
\|( \mathfrak{C}f_{n}-f_{n}(z_{0}))-(\mathfrak{C}f-f(z_{0}))\|_{\infty}^L & \to & 0,\;\text{for all compact  $L$ of {\rm int}(K)}.
\end{eqnarray}
\end{T}
The proof is at Section \ref{PrfThmCty}. The following lemma gives conditions for the weak convergence of the random measures $\mu_{n}$ to a measure $\mu$.
 
 Following \cite{B}, call \emph{convergence-determining class} $\mathcal{C}$ a class of Borel sets such that the weak convergence of measures follows from the pointwise convergence on $\mathcal{C}$. Theorem 2.2 p.15 in \cite{B}  implies that there is a countable such class in $\mathbb{R}^d$.

\begin{Lm}
\label{lm:cvg-random-measures}
Let $\{\mu_{n};\,n \geq 1\}$ be a sequence of random probability measures with characteristic functions $\{\varphi_{n}\,n\geq1\}$. Let $\mu$ be a probability measure on $\mathbb{R}^d$ with characteristic function $\varphi$, assume that one of following holds

\begin{itemize}
\item[$\mathbf{(i)}$] for almost all $h$ of $\mathbb{R}^d$, $ \varphi_{n}(h) \to \varphi(h)$ a.s.,
\item[$\mathbf{(ii)}$] for every $\mu$-continuity Borel set $B$ from a countable convergence-determining class, $\mu_{n}(B)\to \mu(B)$  a.s.,
\end{itemize}
then  $\mu_{n} \Rightarrow \mu$ with probability one.
\end{Lm}

\begin{proof}
$\mathbf{(i)}$: We have
\begin{eqnarray*}
\int_{\mathbb{R}^d} \int_{\Omega}(1- \mathbf{1}_{\{\varphi_{n}(h,\omega) \to \varphi(h,\omega)\}})\mathbb{P}(d \omega)dh=0.
\end{eqnarray*}
Due to Fubini's theorem, with probability one, for almost all $h$ of $\mathbb{R}^d$,
\begin{eqnarray*}
\varphi_{n}(h) \to \varphi(h),
\end{eqnarray*}
and it is well known that it implies the weak convergence of the corresponding probability measures.

$\mathbf{(ii)}$:  Since the class is countable, the pointwise convergences $\mu_{n}(B)\to \mu(B)$ hold simultaneously with probability $1$, and since the class is convergence-determining, it yields the a.s. convergence $\mu_{n}\Rightarrow \mu$.
\end{proof}

\section{Asymptotic rearrangement of random fields}
\label{SecRdmFld}

In this section, we consider a random field $X$ defined on $K_d=[0,1]^d$ and give general results about its asymptotic rearrangement. Then we give the main theorem of convergence in the case of Gaussian fields, in the framework of polygonal approximation. This generalizes the asymptotic convex rearrangement of the Brownian motion derived in Section \ref {sec:Int}.

\subsection{General results}
The notation $\{Y_{n}~;~n\geq 1\}$ stands here for a sequence of smooth vector valued random  fields, and $\{\mu_{n}=\mu_{Y_{n}}~;~n\geq 1\}$ are their distributions. In this section a general result concerning the asymptotics of  $\{\mu_n;\,{n \geq 1}\}$ is given. The objective is to obtain a deterministic limit  measure $\mu$ of the $\mu_n$ , and use the consistency Theorem \ref{ThmCty}. The primary condition for the convergence of $\mu_{n}$ is the convergence of the expectation
\begin{eqnarray}
\label{EquDstMun}
\mathbb{E}(\mu_n(B))=\mathbb{E}\left(\int_{K_d} \mathbf{1}_{Y_{n}(z)\in B}d z\right)=\int_{K_d} \mathbb{P}(Y_n(z) \in B) \textrm{d} z\to \mu(B)
\end{eqnarray}
for some measure $\mu$ and every $\mu$-continuity Borel set $B$.  As a first example, the following proposition gives a sufficient condition on the conjoint laws of the variables $(Y_n(z))_{z \in K_d }$ for the convergence of $\mu_n$.

\begin{T}
\label{ThmCvgGal}Assume that for all $\mu$-continuity Borel sets $B$ in a \emph{convergence-determining} class of $\mathcal{B}_{d}$ (see. \cite{B}, p.15),
\begin{eqnarray}
\label{HypCvgGal}
\int_{(K_d)^2}\sum_{n\geq 1} {\rm{cov}}\left
(\mathbf{1}_{\left\{Y_{n}(z)\in B\right\}},\mathbf{1}_{\left\{Y_{n}(\zeta)\in B\right\}}\right)d z d \zeta < \infty,
\end{eqnarray}
then $\mu_{n}\Rightarrow \mu$ a.s.
\end{T}

\begin{proof}
 For $B$ a $\mu$-continuity Borel set in the convergence determining class,
\begin{eqnarray*}
& &\mathbb{E} \left(| \mu_{n}(B)-\mathbb{E}(\mu_{n}(B))|^2\right)=\mathbb{E}\left(\mu_{n}(B)^2\right)-\left(\mathbb{E}\mu_{n}(B)\right)^2\\
&=&\mathbb{E}\left(\int_{K_d}d z  \mathbf{1}_{Y_{n}(z) \in B}\int_{K_d} d \zeta\mathbf{1}_{Y_{n}(\zeta) \in B}\right)-\int_{K_d}d z\mathbb{P}\left(Y_{n}(z)\in B\right) \int_{K_d} d \zeta \mathbb{P}\left(Y_{n}(\zeta)\in B\right)\\
&=&\int_{K_d^2}d z d \zeta \left[\mathbb{E}\left(\mathbf{1}_{Y_{n}(z)\in B}\mathbf{1}_{Y_{n}(\zeta)\in B}\right)-
\mathbb{E}(\mathbf{1}_{Y_{n}(z)\in B})\mathbb{E}\left(\mathbf{1}_{Y_{n}(\zeta)\in B}\right)\right]\\
&=&\int_{K_d^2}d z d \zeta{\rm{cov}}
 \left(\mathbf{1}_{Y_{n}(z)\in B},\mathbf{1}_{Y_{n}(\zeta)\in B}\right).
\end{eqnarray*}

Hence, hypothesis (\ref{HypCvgGal}), along with Borel-Cantelli's lemma, ensures that with probability one, $\mu_{n}(B) \to \mu(B)$.  Lemma \ref{lm:cvg-random-measures}-$\mathbf{(ii)}$ yields the conclusion.
\end{proof}
For most of the random fields investigated in Section \ref{sec:ex}, 
the covariance
${\rm{cov}}
 \left(\mathbf{1}_{Y_{n}(z)\in B},\mathbf{1}_{Y_{n}(\zeta)\in B}\right)$ is in $O(\frac{1}{n})$ and we cannot have asymptotic rearrangement for $\mathfrak{M}Y_{n}$, but only for a subsequence such that $\frac{1}{\sigma(n)}$ is summable. We need stronger results in this case, and were able to obtain them in the framework of Gaussian fields, interpolated on a simplicial triangulation.

\subsection{Simplicial approximations on $K_d$}
\label{sec:approx}

Most of the commonly investigated random fields of the literature are irregular, and hence cannot be directly rearranged, they need to be approximated by smooth functions.  In this article, we only 
adopted the following paradigm: Given a random real field $X$, define approximations $X_{n}$ of $X$, then normalize and rearrange monotonically their gradient, which will be called $Y_{n}=\frac{1}{b_{n}}\nabla X_{n}$ for some $b_{n}>0$. 

 In this paradigm, one would like the result not to depend on the choice of the approximation $X_{n}$, as long as it converges to $X$. Unfortunately, it is in the very nature of the convex rearrangement to be sensitive to slight changes in the approximation method. Consider for instance the following deterministic example. Define $f_{n}$ as the continuous function on $[0,1]$ null in $0$, linear on each segment $[\frac{k}{n},\frac{k+1}{n}]$ for $1\leq k < n$, and with slope $\pm 1$. Then, $f_n$ uniformly converges to the (convex) null function, but $\mathfrak{C} f_n$ uniformly converges to the convex piece-wise linear function null in $0$ having slope $-1$ on $[0,\frac{1}{2}]$ and $+1$ on $[\frac{1}{2},1]$. To avoid this kind of phenomenon for asymptotic convex rearrangement,  one needs to ensure that the gradient of the approximation resembles the gradient of the original function, or its increments if there is no gradient. That is one of the reasons why we choose for $X_{n}$ the polygonal interpolations of $X$ on the vertices of a triangulation. We present  below the details of the construction.
  
 Call simplex of $\mathbb{R}^d$ the convex hull of any $(d+1)$-tuple of points with non-empty interior. Write $$S_d=\left\{(t_i)_{1 \leq i \leq d}:~ 0 \leq t_i \leq 1, \sum_{i=1}^d t_i\leq 1\right\}$$ the elementary simplex of $\mathbb{R}^d$. Given $z$ in $\mathbb{R}^d$ and  an orthonormal basis $\mathbf{u}=(\mathbf{u}_i)_{1 \leq i \leq d}$ of $\mathbb{R}^d$, define the  simplex with summit $z$, and basis $\mathbf{u}$  as
\begin{eqnarray*}
\Sigma(z,\mathbf{u})=z+\rho_{\mathbf{u}}(S_d),
\end{eqnarray*}
where $\rho_{\mathbf{u}}$ is a linear transformation of $\mathbb{R}^d$ transforming the canonical basis into ${\mathbf{u}}$. 
Any simplex $T$ can be written under such a form, and 
we refer to the ``basis of $T$'' as such a choice of $\mathbf{u}$, and denote it by $\mathbf{u}^T=(\mathbf{u}_{i}^T)_{1 \leq i \leq d}$. Remark that such a choice is not unique.

Call triangulation of $K_d$ any finite simplicial partition of $K_d$. 
For $\mathcal{T}$  such a triangulation, denote by $X^{\mathcal{T}}$ the simplicial approximation of $X$ with respect to $\mathcal{T}$, i.e.  the function which is affine above each $T$ in $\mathcal{T}$ and coincides with $X$ above the vertices of $\mathcal{T}$. We will consider in this paper exclusively approximating triangulations of a special form, described below. Denote by $\mathcal{S}_{\mathcal{T}}$ the finite set of all vectors $u$ of $\mathbb{R}^d$ for which $[z,z+u]$ is the edge of a simplex $T$ of $\mathcal{T}$, for some $z$ in $\mathbb{R}^d$, and by
\begin{equation*}
C_{\mathcal{T}}=\sup_{u\in\mathcal{S}_{\mathcal{T}}}\|u\|
\end{equation*}
the length of the longest edge in $\mathcal{T}$.

Call \emph{germ of triangulation} any finite set of simplices $\mathcal{T}$ verifying the following property. There exists a network $\Gamma$ of $\mathbb{R}^d$ such that
\begin{eqnarray}
\label{PptGrmTri}
 \{\gamma+{T};~\gamma \in \Gamma,T \in \mathcal{T}\} \textrm{  is a partition of $\mathbb{R}^d$.}
\end{eqnarray}
Any network $\Gamma$ satisfying (\ref{PptGrmTri}) is said to be \emph{admissible } for $\mathcal{T}$, and the notation $\Gamma_{\mathcal{T}}$ refers to an arbitrary choice of such a network. Then define, for $n \geq 1$,
\begin{eqnarray*}
\widetilde{\mathcal{T}}_n&=&\bigcup_{T \in \mathcal{T}, \gamma \in \Gamma_{\mathcal{T}}}\left\{\frac{1}{n}(\gamma+T) \cap K_d\right\} .
\end{eqnarray*}
Property (\ref{PptGrmTri}) ensures that $\widetilde{\mathcal{T}}_n$ is indeed a partition  of $K_d$. The problem is that a set $n^{-1}(\gamma+T) \cap K_d$ might not be a simplex if it hits the boundary of $K_{d}$. However, those problematic simplexes
 won't play any role in the asymptotic convex rearrangement because their number is negligible (it is proven later). So, arbitrarily decide of a simplicial partition of each of these simplexes. The result is a triangulation $\mathcal{T}_{n}$ that is a simplicial sub-partition of $\widetilde{\mathcal{T}}_{n}$, and differs from $\widetilde{\mathcal{T}}_{n}$ only regarding the simplices touching the boundary of $K_d$.

Given a finite set of triangles $\mathcal{T}$, denote by $X_{n}^\mathcal{T}=X^{\mathcal{T}_{n}}$ the corresponding approximation of $X$. Since $X_n^{\mathcal{T}}$ is a.e affine, denote by $\nabla X_n^{\mathcal{T}}$ its gradient, defined a.e.. In all the paper, $\{b_n;~ n \geq 1\}$ stands for a sequence of positive numbers which aims to give sense to $\lim_n \frac{1}{b_n} \mathfrak{M} \nabla X_n^{\mathcal{T}}$ (or, equivalently- see Theorem \ref{ThmCty}- to $\lim_n \frac{1}{b_n} \mathfrak{C}  X_n^{\mathcal{T}}$). The renormalized gradient is defined up to a negligible set and is denoted by $$Y_n^{\mathcal{T}}=\frac{1}{b_n} \nabla X_n^{\mathcal{T}}.$$ Using Theorem \ref{ThmCty}, to obtain the rearrangement of $Y_n^{\mathcal{T}}$, it is more convenient to work with its distribution $\mu_n^{\mathcal{T}}=\mu_{Y_n^{\mathcal{T}}}$.

\subsection{Rearrangements of centered Gaussian fields}

\label{SecGssFld}

The specific study of Gaussian fields yields more efficient tools to study the convergence. We give here the statement of the main theorem of this paper, some examples will be derived in the next section to illustrate the theory, for fractional Brownian fields and Brownian sheet.  The generalized Lorenz curve plays a great role in the convex rearrangement of  Gaussian processes, so we introduce it now.

\begin{D}Call $\gamma_{d}$ the $d$-dimensional standard normal distribution. The $d$-dimensional generalized Lorenz curve is 

\begin{eqnarray*}
~GL_{d}=\mathbf{C}_{\gamma_{d}}.
\end{eqnarray*}
\end{D}
In other words, it is the asymptotic convex rearrangement of any field which renormalized gradient measure converges to $\gamma_{d}$. It corresponds in dimension $1$ to the classical Lorenz curve, frequently used in the fields of finance and econometrics.

Approximate a centered Gaussian field $X$ with covariance function $\sigma$ on a germ of triangulation $\mathcal{T}$ by $X_{n}^\mathcal{T}$. The gradient $Y_{n}^\mathcal{T}=b_{n}^{-1}\nabla X_{n}^{\mathcal{T}}$ has the following expression along an edge $[z,z+n^{-1}u]$ of a simplex $T$ of $\mathcal{T}_{n}$,
\begin{equation*}
\langle Y_{n}(z),u\rangle=(n/b_{n})(X(z+n^{-1}u)-X(z)),
\end{equation*}
whence the covariance structure of the gradient field relies on $\mathbb{E}\langle Y_{n}(z),u \rangle \langle Y_{n}(\zeta),v \rangle$ for $[z,z+n^{-1}u]$ and $[\zeta,\zeta+n^{-1}v]$ edges of simplices of $\mathcal{T}_{n}$. An easy computation yields
\begin{equation}
\label{eq: cov-gradient}
\mathbb{E}\langle Y_{n}(z),u \rangle \langle Y_{n}(\zeta),v \rangle=(n/b_{n})^2\sigma^{(2)}_{z,\zeta}(n^{-1}u,n^{-1}v)
\end{equation}
where
\begin{equation*}
\sigma^{(2)}_{z,\zeta}(u,v)=\sigma(z+u,\zeta+v)-\sigma(z+u,\zeta)-\sigma(z,\zeta+v)+\sigma(s,\zeta)
\end{equation*}
is the \emph{local second order increment} of $\sigma$. The following theorem gives a condition for the convergence of $\mathbb{E}\varphi_{n}^\mathcal{T}(h)$, where $h \in \mathbb{R}^d$ and $\varphi_{n}^\mathcal{T}$ is the characteristic function of the image measure $\mu_{n}^\mathcal{T}=\lambda_{d}(Y_{n}^\mathcal{T})^{-1}$.

\begin{T}
\label{thr:cvg-expectation-phi}
Assume that there is a function $\sigma^{\text{diag}}_{z}(u,v),z\in K_{d},\,u,v\in \mathbb{R}_{d}$, continuous in $z$, such that for all $u,v$
\begin{equation}
\label{eq:cvg-sigma-diag}
(n/b_{n})^2\sigma^{(2)}_{z,z}(n^{-1}u,n^{-1}v)\to \sigma^{\text{diag}}_{z}(u,v)
\end{equation}
uniformly in the $z$ where it is defined. For a basis $\mathbf{u}$ and $z\in K_{d}$, denote by $\mu^{z,\mathbf{u}}$ the Gaussian  probability measure on $\mathbb{R}_{d}$ with covariance matrix $(\sigma_{z}^{\text{diag}}(\mathbf{u}_{i}^T,\mathbf{u}_{j}^T))_{ij}$ in basis $\mathbf{u}$, and let $\varphi^{z,\mathbf{u}}$ be its characteristic function. Then $\mathbb{E}\varphi_{n}^\mathcal{T}(h)\to \varphi^\mathcal{T}(h)$, with 
\begin{equation}
\label{eq:limit-expectation-phi}
\varphi^\mathcal{T}(h)=\sum_{T\in \mathcal{T}}\kappa_{T}\int_{K_{d}}\varphi^{z,\mathbf{u}^T}(h)dz,
\end{equation}
where 
\begin{equation*}
\kappa_{T}=\frac{\vol(T)}{\sum_{T\in\mathcal{T}}\vol(T)}.
\end{equation*}
It means that $\varphi^\mathcal{T}$ is the characteristic function of the mixtures of the  $\mu^{z,\mathbf{u}^T}$, $T\in \mathcal{T}$, $z\in K_{d}$.
\end{T}

\begin{proof}
For $T$ in $\mathcal{T}$, denote by 
\begin{equation*}
\mathcal{T}_{n}^T=\{n^{-1}(\gamma+T)\in \mathcal{T}_{n} \},
\end{equation*}
 all the simplices of $\mathcal{T}_{n}$ obtained by translation and rescaling of $T$. 
We have, for $h$ in $\mathbb{R}^d$,
\begin{align*}
\varphi_{n}^\mathcal{T}(h)&=\int_{K_{d}}\exp(\imath \langle Y_{n}(z), h\rangle)=\sum_{S\in \mathcal{T}_{n}}\int_{S}\exp(\imath \langle Y_{n}(z), h \rangle) dz\\
& =\sum_{S\in \mathcal{T}_{n}} \vol(S)\exp(\imath \langle Y_{n}(S), h \rangle)
\end{align*}
where $Y_{n}(S)$ stands for the common value of $Y_{n}$ over $S$. 

Let $T=\Sigma(z_{T},\mathbf{u}^T)$ be a simplex of $\mathcal{T}$. If we put
\begin{equation*}
\varphi_{n}^T(h)=\sum_{S\in \mathcal{T}_{n}^\mathcal{T}}\vol(S)\exp(\imath \langle Y_{n}(S),h\rangle),
\end{equation*}
we have $\varphi_{n}=\sum_{T\in \mathcal{T}}\varphi_{n}^T+c_{n}$, where $c_{n}$ is the integral over the area where simplices of $\mathcal{T}_{n}$ touches the border. It is clear that $-C_{\mathcal{T}}/n \leq c_{n}\leq C_{\mathcal{T}}/n$, whence $c_{n}\to 0$ a.s..
For  $S=n^{-1}(\gamma+T)=\Sigma(z_{S},n^{-1}\mathbf{u}^T)$  a simplex of $\mathcal{T}_{n}$ we have, by (\ref{eq: cov-gradient}),
\begin{equation*}
\mathbb{E}\langle Y_{n}({S}), \mathbf{u}_{i}^T\rangle\langle Y_{n}(S),\mathbf{u}_{j}^T\rangle=(n/b_{n})^2\sigma^{(2)}_{z_{S},z_{S}}(n^{-1}\mathbf{u}_{i}^T,n^{-1}\mathbf{u}_{j}^T).
\end{equation*}
For $z$ in $K_{d}$, denote by $z_{n}^T$ the closest point such that $S=\Sigma(z_{n}^T,n^{-1}\mathbf{u}^T)$ is a simplex of $\mathcal{T}_{n}^T$, and let $\varphi_{n}^{z}(h)$ be the characteristic function of $Y_{n}(z_{n}^T)$.
Using hypothesis (\ref{eq:cvg-sigma-diag}), the expectation of this function converges pointwise (in $z$) to $\varphi^{z,\mathbf{u}^T}(h)$
and is bounded by $1$. Thus, by denoting $\bar S=\{z\in K_{d}:\, z_{n}^T \in S\}$ for $S\in \mathcal{T}_{n}^T$, we have
\begin{equation*}
\mathbb{E}\sum_{S\in \mathcal{T}_{n}^T}\vol(\bar S)\varphi_{n}^{z_S}(h)=\mathbb{E}\int_{K_{d}}\varphi_{n}^z(h)dz\to \int_{K_{d}}\varphi^{z,\mathbf{u}^T}(h)dz.
\end{equation*} 
Thus
\begin{equation*}
\mathbb{E}\varphi_{n}^T(h)=\mathbb{E}\sum_{S\in \mathcal{T}_{n}^T}\vol(S)\varphi_{n}^{z_S}(h)\to \kappa_{T} \int_{K_{d}}\varphi^{z,\mathbf{u}^T}dz.
\end{equation*}

Summing over $T\in\mathcal{T}$ gives the result.

\end{proof}

Thus the candidate for the limit, given by (\ref{eq:limit-expectation-phi}), is known, provided (\ref{eq:cvg-sigma-diag})  is satisfied.
We state now the main theorem of this paper, which gives a more efficient condition for the weak  convergence of $\mu_n^\mathcal{T}$ than Theorem \ref{ThmCvgGal}. 
\begin{T}
\label{CorCvgGss}
Keeping the previous notation, we have
\begin{equation}
\label{eq:bound-4thmoment-phi}
\mathbb{E}|\varphi_{n}^\mathcal{T}(h)-\mathbb{E}\varphi_{n}^\mathcal{T}(h)|^4\leq C \left((n/b_{n})^2\sum_{S,S'\in\mathcal{T}_{n}}\vol(S)\vol(S')|\sigma^{(2)}_{z,\zeta}(n^{-1}u,n^{-1}v)|\right)^2\end{equation}
for some constant $C>0$, where $[z,z+n^{-1}u]$ and $[\zeta,\zeta+n^{-1}v]$ are edges of $S$ and $S'$, respectively.
\end{T}

The proof is at section \ref{PrfThmCvgMunGss}.
In all our examples, we have the summability of the right hand term, which gives us the a.s. weak convergence of $\mu_{n}^{\mathcal{T}}$.

\section{Examples}
\label{sec:ex}

 \subsection{Fractional Brownian field}
 The fractional Brownian field is a celebrated model that includes many other famous random fields and processes, such as the fractional Brownian motion or the L\'evy field. For $\alpha \in (0,2)$,  the Fractional Brownian field is the unique centered Gaussian field $X^{\alpha}$ which covariance function is, up to a constant, 
\begin{equation*}
\sigma(z,\zeta)=(\|z\|^\alpha+\|\zeta\|^\alpha-\|z-\zeta\|^\alpha).
\end{equation*}

\begin{T}
\label{thr:cvg-fractional}
Let $\mathcal{T}$ be a germ of triangulation, and define
\begin{align*}
b_{n}&=n^{1-\alpha/2}, \\
Y_{n}^{\alpha,\mathcal{T}}&=b_{n}^{-1}\nabla X_{n}^{\alpha,\mathcal{T}},\\
\mu_{n}^{\alpha,\mathcal{T}}&=\lambda_{d}(Y_{n}^{\alpha,\mathcal{T}})^{-1}.
\end{align*}
We have the convergence
\begin{equation*}
\mu_{n}^{\alpha,\mathcal{T}}\Rightarrow \mu^{\alpha,\mathcal{T}}=\sum_{T\in\mathcal{T}}\kappa_{T}\mu^{\alpha,\mathbf{u}^T}\,a.s.,
\end{equation*}
where $\mu^{\alpha,\mathbf{u}}$ is a Gaussian probability measure with covariance matrix
\begin{equation*}
\Lambda^{\alpha,\mathbf{u}}_{ij}=\|\mathbf{u}_{i}^T\|^\alpha+\|\mathbf{u}_{j}^T\|^\alpha-\|\mathbf{u}_{i}^T-\mathbf{u}_{j}^T\|^\alpha
\end{equation*}
in basis $\mathbf{u}$.
We have also 
\begin{eqnarray*}
\frac{1}{b_{n}}\mathfrak{M}\nabla X_{n}^{\alpha,\mathcal{T}}\to \mathbf{M}_{\mu^{\alpha,\mathcal{T}}},\;\frac{1}{b_{n}} \mathfrak{C}X_{n}^{\alpha,\mathcal{T}}\to\mathbf{C}_{\mu^{\alpha,\mathcal{T}}},
\end{eqnarray*}  
in the sense of Theorem \ref{ThmCty}.
\end{T}

\begin{proof}
Since $\alpha$ is fixed, we omit int the proof exponent $\alpha$ for the sake of clarity. We have for $z$ in $K_{d}$, $u,v\in \mathbb{R}^d$,
\begin{equation*}
\sigma^{(2)}_{z,z}(u,v)=\|u\|^\alpha+\|v\|^\alpha-\|u-v\|^\alpha 
\end{equation*}
whence (\ref{eq:cvg-sigma-diag}) is satisfied with $\sigma^{\text{diag}}_{z}(u,v)=\sigma^{(2)}_{z,z}(u,v)$. It follows from Theorem \ref{thr:cvg-expectation-phi} that 
\begin{equation*}
\mathbb{E}\varphi_{n}^\mathcal{T}(h)\to  \sum_{T\in \mathcal{T}}\kappa_{T}\varphi^{\mathbf{u}^T}
\end{equation*}
where $\varphi^{\mathbf{u}^T}$ has covariance matrix $\Lambda^{\mathbf{u}^T}$.

Thus, for any germ of triangulation $\mathcal{T}$, we have by (\ref {eq:bound-4thmoment-phi}), with $n/b_{n}=n^{\alpha/2}$,
\begin{align}
\label{eq:4thmoment-fractional}
\mathbb{E}&|\varphi_{n}^{\mathcal{T}}(h)-\mathbb{E}\varphi_{n}^{\mathcal{T}}(h)|^4 = O\left(n^\alpha\sum_{(S,S')\in\mathcal{T}_{n}^2}\vol(S)\vol(S')|\sigma^{(2)}_{z,\zeta}(n^{-1}u,n^{-1}v)|\right)^2
\end{align}
where $[z,z+n^{-1}u]$ and  $[\zeta,\zeta+n^{-1}v]$ are edges of resp. $S$ and $S'$, hence satisfy $\|u\|,\|v\|\leq C_{\mathcal{T}}$. 
 We put $t_{n}= (C_{\mathcal{T}}+1) n^{-1}$. We distinguish the set $\Sigma_{n}$ of pairs $(S,S')$ of $\mathcal{T}_{n}^2$ that are at distance more than $t_{n}$ from $\Theta$ in $K_{d}^2$, and the other ones, which contribution is, since $\sigma$ is $\alpha$-Holder, in
\begin{equation*}
n^{\alpha}\sum_{(S,S')\in \mathcal{T}_{n}^2\setminus \Sigma_{n}}\vol(S)\vol(S')t_{n}^\alpha \leq n^{\alpha}\vol(\Theta+B(0,t_{n}))t_{n}^\alpha=O(n^\alpha n^{-1-\alpha}),
\end{equation*}
 whence this term is square summable.

 In view of using (\ref{eq:4thmoment-fractional}), for $\|u\|,\|v\|\leq C_{\mathcal{T}}$, we have
\begin{align}
\label{eq:bound-sigma2-fractional}
\nonumber\sigma^{(2)}_{z,\zeta}&(n^{-1}u,n^{-1}v)\\
=&\|z-\zeta\|^\alpha+\|z-\zeta+n^{-1}(u-v)\|^\alpha-\|z-\zeta+n^{-1}u\|^\alpha-\|z-\zeta-n^{-1}v\|^\alpha\\
=&n^{-2}O(\|z-\zeta\|^{\alpha-2}).
\end{align}
Thus we have, 
 \begin{align}
 \label{eq:2d-term-fractional}
n^\alpha\sum_{(S,S')\in \Sigma_{n}}\vol(S)\vol(S')&\sigma^{(2)}_{z,\zeta}(n^{-1}u,n^{-1}v)\\
=&O(n^{\alpha-2}\sum_{(S,S') \in \Sigma_{n}}\vol(S)\vol(S')\|z-\zeta\|^{\alpha-2}).
\end{align}
Remark that at fixed $z$, the sum $\sum_{S': (z,\zeta)\in\Sigma_{n}}\|z-\zeta\|^{\alpha-2}$ is smaller than $\sum_{S': t_{n}\leq \|\zeta\| \leq 1}\|\zeta\|^{\alpha-2}$, where the sum is over all $S'$ that are of the form $n^{-1}(\gamma+T)$ for $T$ in $\mathcal{T}$ and $\gamma$ in $\Gamma_{\mathcal{T}}$ (and not only those of $\mathcal{T}_{n}^T$ that intersect $K_{d}$), but with summit $\zeta$ that has norm in $[t_{n},1]$.
The function defined on $\{\zeta \in \mathbb{R}^d: t_{n}\leq \|\zeta\| \leq 1\}$ by
\begin{equation*}
\beta(z)=\|\zeta\|^{\alpha-2}\text{ for $z$ belonging to $S'$ (which has summit $\zeta$)}
\end{equation*} is smaller than 
\begin{equation*}
\bar\beta(z)=(\|z\|-C_{\mathcal{T}}/n)^{\alpha-2}
\end{equation*}
because $\zeta$ has norm larger than $\|z\|+C_{\mathcal{T}}/n$, given that $z,\zeta\in S'$ and $S'$ has diameter smaller than $C_{\mathcal{T}}/n$, and $\alpha-2\leq 0$. Whence
\begin{align*}
\sum_{S':t_{n}\leq\|\zeta\| \leq 1}\vol(S')\|\zeta\|^{\alpha-2}&\leq  \int_{t_{n}-C_{\mathcal{T}}/n \leq \|\zeta\|\leq 1}\|z\|^{\alpha-2}dz\\
=&O\left(\int_{t_{n}-C_{\mathcal{T}}/n\leq r \leq 1}r^{\alpha-2}r^{d-1}dr\right)\\
 =&O({n}^{-\alpha-d+2}).
\end{align*}
Finally the term (\ref{eq:2d-term-fractional}) is in $n^{\alpha-2}n^{-\alpha-d+2}$, whence it is square summable for $d\geq 1$, and the  sum in (\ref{eq:4thmoment-fractional}) is finite. Thus by Borel Cantelli's lemma $\varphi_{n}^\mathcal{T}(h)-\mathbb{E}\varphi_{n}^\mathcal{T}(h)\to 0$ a.s., whence Lemma \ref{lm:cvg-random-measures}-$\mathbf{(i)}$ brings the conclusion.
\end{proof}

Theorem  \ref{thr:cvg-fractional} retrieves the convergence of the $1$-dimensional fractional Brownian motion $X^\alpha$ interpolated on $\{k/n;\,k=0,1,2,\dots,n\}$,
\begin{equation*}
n^{\alpha/2-1}\RgC X_{n}^\alpha(z)) \to GL_{1}(z),\,z\in(0,1).
\end{equation*}
This result was already present in \cite{DT}, who furthermore obtained uniform convergence on $[0,1]$.

The asymptotic rearrangement is consistent under the action of rotations: Indeed, if $\mu^{\mathcal{T}}$ is the limit measure with germ of triangulation $\mathcal{T}$, we have for all rotation $\rho$ and germ of triangulation $\mathcal{T}$, $\mu^{\rho(\mathcal{T})}=\mu^\mathcal{T}\rho^{-1}(\cdot)$. This is due to the isotropy of the field, and will not be the case in the subsequent examples.

\subsection{Brownian sheet}

This section is devoted to the study of the Brownian sheet, another irregular centered Gaussian field. 
For $z$ and $\zeta$ two elements of $\mathbb{R}^d$, denote by $z\wedge \zeta$ the vector whose  coordinates are the pointwise minimum coordinates of $z$ and $\zeta$, and $\underline{z}$ is the product of coordinates of $z$.
 The Brownian sheet  is defined on $(\mathbb{R}_{+})^d$ as the Gaussian field with covariance function $\sigma(z,\zeta)=\underline{z \wedge \zeta}$.  Here, we use the notation of Section \ref{sec:approx}, where $X$ is a Brownian sheet.
\begin{T}
\label{ThmCvgCtv}
Let  $T=\Sigma(0,\mathbf{u})$ be a  simplex of $(\mathbb{R}_+)^d$. We define
\begin{eqnarray}
\label{EqBas}
\nonumber & & \mathbf{u}_{i,j}=\mathbf{u}_i \wedge \mathbf{u}_j-\mathbf{u}_i \wedge 0- \mathbf{u}_j \wedge 0 \in \mathbb{R}^d, ~i,j \in \{1,\dots ,d\},\\
\label{Deflz}
& &l(z)=(z_2\ldots z_d, z_1z_3 \ldots z_d , \ldots,z_1  \ldots z_{d-1}),~z \in K_d,~\\
\nonumber  & & {\Lambda}^{\mathbf{u}}(z)_{i,j}=\langle l(z), \mathbf{u}_{i,j}\rangle,~i,j \in \{1,\dots ,d\}.
\end{eqnarray}
We cal  $\varphi^{\mathbf{u}}$ the characteristic function of the Gaussian probability measure with covariance matrix $\Lambda^{\mathbf{u}}$ in basis $\mathbf{u}$. We have, for every $h$ in $\mathbb{R}^d$, and $b_{n}=\sqrt{n}$
\begin{eqnarray*}
\varphi_n^\mathcal{T} (h)\to \varphi^\mathcal{T}(h) =\sum_{T\in \mathcal{T}}\kappa_{T} \varphi^{\mathbf{u}_{T}}(h)\quad\text{a.s.},
 \end{eqnarray*}
 whence $\mu_{n}^\mathcal{T}\Rightarrow\mu^{\mathcal{T}}$, the measure whose characteristic function is $\varphi^\mathcal{T}$. We have also, in virtue of Theorem \ref{ThmCty} 
\begin{equation*}
\frac{1}{\sqrt{n}}\RgC X_{n}^\mathcal{T}\to \mathbf{C}_{\mu^\mathcal{T}}, \frac{1}{\sqrt{n}}\RgM \nabla X_{n}\to \mathbf{M}_{\mu^\mathcal{T}}\text{ almost surely}.
\end{equation*}
\end{T}

\begin{proof}

We  use Theorem \ref{thr:cvg-expectation-phi} to compute the only possible limit and Theorem \ref{CorCvgGss} to show the almost sure convergence.
Let $z\in K_{d}$ and $u,v\in\mathbb{R}^d$. We have
\begin{align*}
(n/b_{n})^2&\sigma^{(2)}_{z,z}(n^{-1}u,n^{-1}v)\\
&=n\left(\un{(z+\frac{1}{n}u)\wedge (z+\frac{1}{n}v})-\un{(z+\frac{1}{n}u)\wedge z}-\un{z \wedge (z+\frac{1}{n}v)}+\un{z}\right)\\
&=n\left(\un{(z+\frac{1}{n} u \wedge v)}-\un{(z+\frac{1}{n}u \wedge 0)}-\un{(z+\frac{1}{n}v \wedge 0)}+\un{z}\right).\\
\end{align*}
Consider now the function $\Pi$ on $\mathbb{R}^d$ defined by $\Pi(z)=\un{z}$. It admits, for all $z,h \in \mathbb{R}^d$, the development
\begin{eqnarray*}
\Pi(z+h)=(z_1+h_1)\dots (z_d+h_d)=\Pi(z)+\langle l(z),h\rangle+q(z,h),
\end{eqnarray*}

where
\begin{eqnarray*}
l(z)&=&(z_2\dots z_d, z_1z_3\dots z_d ,\dots  z_1\dots z_{d-1}),\\
q(z,h)&\leq &C\|h\|^2
\end{eqnarray*}
for some  constant $C$.
Hence (\ref{eq:cvg-sigma-diag}) is satisfied with
\begin{eqnarray*}
\sigma^{\text{diag}}_{z}(u,v)=\langle l(z), u\wedge v-u\wedge 0-v\wedge 0\rangle.
\end{eqnarray*}

 For $I\subseteq \{1,2,\dots,d\}$, define $\varphi_{I}(z)=\prod_{i\in I}z_{i}$. Let $I$ be the set of indices for which $z_{i}>\zeta_{i}$, and $I^c$ its complementary in $\{1,\dots,d\}$. Take $z,\zeta$ in $K_{d}$ with distinct coordinates and $u,v$ such that $\|u\|_{\infty},\|v\|_{\infty}<\inf_{i}|z_{i}-\zeta_{i}|$.
\begin{align*}
\sigma^{(2)}_{z,\zeta}(u,v)&=\sigma(z,\zeta)-\sigma(z,\zeta+v)+\sigma(z+u,\zeta+v)-\sigma(z+u,\zeta)\\
&=\varphi_{I}(z)(\varphi_{I^c}(\zeta)-\varphi_{I^c}(\zeta+v))+\varphi_{I}(z+u)(\varphi_{I^c}(\zeta+v)-\varphi_{I^c}(\zeta))\\
&=(\varphi_{I}(z)-\varphi_{I}(z+u))(\varphi_{I^c}(\zeta)-\varphi_{I^c}(\zeta+v)).
\end{align*}
Since the $\varphi_{I}$ are of class $\mathcal{C}^1$ on $K_{d}$, there is a  constant $C$ such that
\begin{equation*}
\sigma^{(2)}_{z,\zeta}(u,v)\leq C\|u\|\|v\|
\end{equation*}
whenever $\|u\|_{\infty},\|v\|_{\infty}<\inf_{i}|z_{i}-\zeta_{i}|$.

Thus we define the class $\Sigma_{n}$ of simplices $S,S'$ of $\mathcal{T}_{n}$ for which every $z\in S,\zeta\in S'$ satisfy
\begin{equation*}
|z_{i}-\zeta_{i}|>C_{\mathcal{T}}/n.
\end{equation*}
It follows that the sum (\ref{eq:bound-4thmoment-phi}) is divided in two terms, one with the sum over $\Sigma_{n}$, and the rest. The sum over $\Sigma_{n}$ is clearly in $O((n/b_{n})^2n^{-2}\vol(K_{d}))=O(n^{-1})$, hence square summable, and the rest is majorized by the volume in $K_{d}^2$ of all points $(z,\zeta)$ that satisfy $\inf_{i}|z_{i}-\zeta_{i}|\leq C_{\mathcal{T}}/n$, hence is in $O(1/n)$, and is square summable too. Thus (\ref{eq:bound-4thmoment-phi}) is summable, and by Borel Cantelli's lemma we have the result.

\end{proof}

Finding the expression of $\mathbf{C}_{\mu^\mathcal{T}}$ is not an easy task, and in general we were not  able to derive explicit formulas. We present here a tractable expression for the $2$-dimensional Brownian sheet with the germ of triangulation $\mathcal{T}_{0}=\{\Sigma(0,\mathbf{e}),\Sigma(0,-\mathbf{e})\}$.

 With the notation of Theorem \ref{ThmCvgCtv}, we have
\begin{eqnarray*}
\mathbf{e}_{1,1}&=&\mathbf{e}_{1},\\
\mathbf{e}_{1,2}&=&\mathbf{e}_{2,1}=0,\\
\mathbf{e}_{2,2}&=&\mathbf{e}_{2},\\
(-\mathbf{e})_{1,1}&=&\mathbf{e}_{1},\\
(-\mathbf{e})_{1,2}&=&(-\mathbf{e})_{2,1}=0,\\
(-\mathbf{e})_{2,2}&=&\mathbf{e}_{2}.
\end{eqnarray*}

 We are looking for the expression of the asymptotic convex rearrangement $\mathbf{C}_{\mu^{\mathcal{T}_0}}$, which gradient distribution is the measure
\begin{equation*}
\mu_{\mathcal{T}_{0}}(B)=\int_{K_{2}}\mu_{x,y}(B)dxdy, 
\end{equation*}
where, according to (\ref{EqBas}), $\mu_{x,y}$ is Gaussian with covariance matrix
\begin{equation*}
{\Lambda}^{\mathbf{e}}(x,y)=\left(\begin{array}{cc}y & 0 \\0 & x\end{array}\right).
\end{equation*}

Let $C_{a,b}=(-\infty,a]\times (-\infty,b]$ be an infinite rectangle of $\mathcal{B}_{2}, a,b \in \mathbb{R}$. We have
\begin{eqnarray}
\label{EquCarPsi}
\mu_{\mathcal{T}_{0}}(C_{a,b})=\int_{K_{2}}dxdy\int_{C_{a,b}}d h_{1}dh_{2} \frac{\exp(-\frac{1}{2}({h_{1}^2/ y+h_{2}^2/x}))}{2\pi\sqrt{xy}}&=&G(a)G(b),
\end{eqnarray}
where
\begin{eqnarray*}
G(a) = \int_{-\infty}^a  d h \int_{0}^1  \frac{\exp(-\frac{h^2 }{2x})}{\sqrt{2 \pi x}}d x ,~a \in \mathbb{R}.
\end{eqnarray*}
It is a non-decreasing bijection from $\mathbb{R}$ to $[0,1]$.
In consequence, we define $\mathbf{C}_{\mu^{\mathcal{T}_0}}$ by

\begin{eqnarray*}
\psi(x)=\int_{0}^x G^{-1}(t) \ dt,\\
\mathbf{C}_{\mu^{\mathcal{T}_0}}(x,y)=\psi(x)+\psi(y).
\end{eqnarray*}
Since $\psi$ is convex, so is $\mathbf{C}_{\mu^{\mathcal{T}_0}}$. We have

\begin{eqnarray*}
\mu_{\nabla \mathbf{C}_{\mu^{\mathcal{T}_0}}}(C_{a,b})&=&\int_{K_{2}}\mathbf{1}_{\{\nabla C_{\mu}^{\mathcal{T}_{0}}(z)\in C_{a,b}\}}dz\\
&=&\int_{K_2} \mathbf{1}_{\{\psi'(x) \leq a\}} \mathbf{1}_{\{\psi'(y) \leq b\}} d xd y=\int_{K_2} \mathbf{1}_{\{x \leq G(a)\}} \mathbf{1}_{\{y \leq G(b)\}} d xd y\\&=& G(a)G(b).
\end{eqnarray*}
$\mathbf{C}_{\mu^{\mathcal{T}_0}}$ indeed has gradient distribution  (\ref{EquCarPsi}).
This function is represented on Figure \ref{fig:brw-sheet}.

\section{Discussion}
\label{SecDsc}

In this article we developed  tools for computing the asymptotic convex rearrangements of some random fields. We observed that there was a strong dependency on the choice of the triangulation used for approximating the field. In \cite{DT}, it becomes apparent that for  some $1$-dimensional Gaussian processes, the Lorenz curve seems a ``universal'' asymptotic convex rearrangement, in the sense that it is the same for polygonal and convoluted approximations.

In the multivariate case, the anisotropy of some fields make this universality impossible.
If $\mu^{\mathcal{T}}$ is the limit measure, and $\rho$ is a rotation of $\mathbb{R}^d$, measures $\mu^{\rho(\mathcal{T})}\rho(\cdot)$ and $\mu^\mathcal{T}$ are in general different, unless the field is isotropic.
The mapping that associates to each rotation $\rho$ its action $\mu\mapsto \mu^{\rho(\tau)}\rho(\cdot)$ can alternatively serve to measure the anisotropy.

\section{Proofs}
\label{SecPrf}

\subsection{Proof of Theorem \ref{ThmCty}}
\label{PrfThmCty}
Without loss of generality, we suppose $f_{n}$ and $f$ convex. It allows us to omit $\mathfrak{M}$ and $\mathfrak{C}$ in the writing.

\textsl{(\ref{EquCty1})$\Rightarrow$ (\ref{eq:CvgMsr}):} 
 Let $\varphi$ be a bounded real continuous function which support lies in a compact $L \subseteq \inr(K)$. Since $\nabla f_{n}\to \nabla f$ for the $L_{1}$ norm, $r$

 Assume first that we have the $L^1$ convergence of $\nabla f_{n}$ to $\nabla f$ on all $K$.
The family $\{\mu_{n};\,n\geq 1\}$ is tight. Indeed, denote by $B_{1}(0,M)$ the ball of radius $M$ for the $\|\cdot\|_{1}$ norm in $\mathbb{R}^d$. Markov's inequality yields, for $M\geq 0$,
\begin{equation}
\mu_{n}(B_{1}(0,M)^c)\leq \frac{1}{M}\int_{K}\|\nabla f_{n}\|_{1}.
\end{equation}
The $L^1$ convergence of  $\nabla f_{n}$ implies that the right hand member converges to $\frac{1}{M}\int_{K}\|\nabla f\|_{1}<\infty$. From there, for all $\varepsilon>0$, there is $M\geq 0$ such that, for sufficiently large $n$, $\mu_{n}(B_{1}(0,M)^c)\leq\varepsilon$, which proves the tightness. To conclude, we need to show that the only possible limit of all convergent sub-sequence of $\{\mu_{n}\}$ is  $\mu$. Let ${\mu_{n'}}$ be a subsequence that converges to a measure $\mu'$. Since $\nabla f_{n'}$ converges for the $L^1$ norm to $\nabla f$, according to the converse of Lebesgue Theorem, there is a subsequence  $\nabla f_{n''}$ that converges to   $\nabla f$ a.e..  Thus, for every continuous function with compact support    $\varphi$ on $\mathbb{R}^d$, $\int_{K} \varphi(\nabla f_{n''})\to \int_{K} \varphi(\nabla f)$, which means $\int_{\mathbb{R}^d}\varphi(x)\mu_{n''}(dx)\to \int_{\mathbb{R}^d}\varphi(x)\mu(dx)$. Since $\mu_{n''}\Rightarrow \mu'$, it follows that $\mu'=\mu$, whence $\mu_{n}\Rightarrow \mu$.

Let us treat now the general case, where we only have the $L^1$-convergence on each compact  of $\inr(K)$. We consider a non-decreasing family of compacts $\{K_{\varepsilon};\,\varepsilon>0\}$ whose union is $\inr(K)$. The convergence holds on every $K_{\varepsilon},\varepsilon>0$. Denote, for a function $u$ on $K$, by $u^\varepsilon$ its restriction to $K_{\varepsilon}$. Put $\mu_{n}^\varepsilon$ the image of Lebesgue measure under $\nabla f_{n}^\varepsilon$, and $\mu^\varepsilon$ that of $\nabla f^\varepsilon$. From what we just proved, $\mu_{n}^\varepsilon\Rightarrow \mu^\varepsilon$ for every $\varepsilon>0$. Let now $B$ be a Borel set of $\mu$-continuity in $\mathbb{R}^d$. It remains to show that $\mu_{n}(B)\to \mu(B)$. Since $B$ is also a $\mu^\varepsilon$-continuity set ($\mu^\varepsilon \leq \mu$), we have $\mu_{n}^\varepsilon(B)\to \mu^\varepsilon(B)$. Then
\begin{equation}
|\mu_{n}(B)-\mu(B)|\leq |\mu_{n}^\varepsilon(B)-\mu^\varepsilon(B)|+\lambda_d(K_{\epsilon}^c),
\end{equation}
the result comes by letting $\varepsilon$ go to $0$. \\



\textsl{(\ref{EquCty2}) implies (\ref{EquCty1}):} We present the result under the form of a lemma, that is also useful later.

\begin{Lm}
\label{LmConMon}
Let $K$ be a compact convex set, and $\{f_{n};\,n\geq 1\}$ a sequence of convex functions that  converge pointwise to a convex continuous function $f$ on $K$. Then $\nabla f_{n}$ converges to $\nabla f$ for the $L^1$ norm on each convex compact subset of $ \mathop{\textrm{int}}(K)$.
\end{Lm}


\begin{proof}[Proof of Lemma \ref{LmConMon}]
We prove the lemma in three steps.

\textbf{Equilipschitz convex functions on $[0,1]$:}
For $\kappa>0$, let $\mathcal{C}_{\kappa}$ be the class of $\kappa$-Lipschitz convex functions on $[0,1]$. Assume that  $f$ and the $(f_{n})$ are in $\mathcal{C}_{\kappa}$. Pick a dense countable subset $S=\{x_{k},k \in \mathbb{N}\}$ in $[0,1]$. Since the $f_{n}'$ are bounded (by $\kappa$), by the diagonal sub-sequence method, we can find 
a sub-sequence $f'_{\sigma(n)}$ such that, for all $k$, $f'_{\sigma(n)}(x_{k})$ converges to some value $g(x_{k})$, where $g$ is non-decreasing on $S$. Call also $g$ its  unique right-continuous non-decreasing continuation on $[0,1]$. Let $x$ be a continuity point of $g$ and $\epsilon>0$. Then, let $y\leq z$ be in $S$ such that $0 \leq g(z)-g(y)\leq \epsilon$ and $y\leq x \leq z$. For $n$ large enough, since the $f_{n}'$ are non-decreasing,
\begin{eqnarray*}
-2\epsilon\leq g(x)-g(z)+g(z)-f_{n}'(z) \leq g(x)-f_{n}'(x)\\
\leq g(x)-g(y)+g(y)-f_{n}'(y) \leq 2\epsilon.
\end{eqnarray*}
Hence $f_{n}'$ converges to $g$ in each of its continuity points, i.e almost everywhere according to Riesz-Nagy theorem. Since $g$ is bounded (by $\kappa$), $f_{n}'$ converges to $g$ for the $L^1$ norm, by Lebesgue theorem. By integration, $g$ equals $f'$ a.e and  we have the result.\\

\textbf{Convex functions on [0,1]:}
Drop  the assumption that the $f_{n}$ are equilipschitz. Let $I=[a,b]$ be a compact subinterval of $]0,1[$. Then, for each $f_{n}$, for any $x$ in $I$, we have, by convexity,
\begin{eqnarray*}
\frac{f_{n}(a)-f_{n}(0)}{a} \leq f_{n}'(x) \leq \frac{f_{n}(1)-f_{n}(b)}{1-b}.
\end{eqnarray*}
Since the left and right hand terms converge to finite values as $n$ goes to $\infty$, the $f_{n}$ are equilipschitz on $I$, and using the previous result, $f_{n}'$ converges to $f'$ for the $L^1$ norm on $I$. \\

\textbf{Convex functions on K:}
Let $I_{i}, 1\leq i \leq d$, be compact intervals of $\mathbb{R}$ such that $C=I_{1}\times \dots  \times I_{d}$ is a compact rectangle contained in $ \mathop{\textrm{int}}(K)$. Take $i$ in $\{1,2,\dots,d\}$.  For $z$ in $I_{1}\times \dots \times \widehat{I_{i}}\times \dots \times I_{d-1}$ (meaning $I_{i}$ is removed from the product), denote by $I_{z}$
 the maximal segment of $C$ with direction $\mathbf{e}_i$ containing $z$. Define  $$G_{n,z,i}(x)=\langle\nabla f_{n}(z,x)-\nabla f(z,x),\mathbf{e}_i \rangle$$ where $x$ is a $1$-dimensional parameter such that $(z,x)$ describes $I_{z}$, and $C_{n,i}(z)=\|G_{n,z,i}\|_{L^1}^{I_{z}}$. Now we have, with Fubini's theorem,
\begin{eqnarray*}
\|\langle \nabla f_{n}-\nabla f,\mathbf{e}_i\rangle \|_{L^1}^C = \int_{I_{1}\times \dots\widehat{I_{i}}\dots \times I_{d-1}}C_{n,i}(z) d z,
\end{eqnarray*}
whence
\begin{equation*}
\|\nabla f_{n}-\nabla f\|_{L^1}^C=\sum_{i=1}^d \int_{I_{1}\times \dots\widehat{I_{i}}\dots \times I_{d-1}}C_{n,i}(z) d z.
\end{equation*}
Let $1\leq i \leq d$ and  $z$ in $I_{1}\times\dots \widehat{I}_{i}\dots I_{d}$.  Since $f_{n}$ uniformly converges to $f$ on $K$, it also does on a segment $J_{z}$ which interior contains $I_{z}$. 
The restriction of $f_{n}$ to $I_{z}$ is hence the restriction of a $1$-dimensional convex function that converges uniformly to the convex function $f$ on $J_{z}$, and this case has been treated in the second part of the proof. Thus, each integrand ${C}_{n,i}(z)$ converges pointwise to $0$.

 To dominate it, we write $I_{z}=:[a_{z},b_{z}]$, and call $c_{z}$ a point in $I_{z}$ where the monotone function $\langle\nabla f_{n}(z,\cdot),\mathbf{e}_i\rangle$ reaches $0$, or $c_{z}=a_{z}$ (arbitrarily) if $0$ is not reached. Then, using the monotonicity of $\langle\nabla f_{n}(z,\cdot),\mathbf{e}_i\rangle$, we have 

\begin{align*}
 C_{n,i}(z) &\leq \| f_{n}(a_{z})\|+\| f_{n}(b_{z})
\|+2\| f_{n}(c_{z})\|+\|\langle\nabla f,\mathbf{e}_i\rangle\|_{L^1}^{I_{z}} \\
&\leq 4  \| f_{n} \|_{\infty}^{C} +\|\langle \nabla f,\mathbf{e}_i\rangle\|_{L^1}^{I_{z}}\\
& \leq 4\| f \|_{\infty}^{C}+\|\langle\nabla f,\mathbf{e}_i\rangle\|_{L^1}^{I_{z}}+o(1).
\end{align*}
The last upper bound is due to the fact that the pointwise convergence of $f_{n}$ to $f$ on the convex $C$ yields uniform convergence.  $\|\langle\nabla f,\mathbf{e}_i\rangle\|_{L^1}^{I_{z}}$ is integrable because $\nabla f$ is integrable, and Lebesgue's theorem gives us the conclusion
\begin{eqnarray*}
\|\nabla f_{n}-\nabla f\|_{L^1}^C \to 0.
\end{eqnarray*}
Now, each convex compact subset of $ \mathop{\textrm{int}}(K)$ is contained in a finite union of such rectangles, and we have the conclusion.
\end{proof}

\textsl{Proof of $(\ref{eq:CvgMsr}) \Rightarrow (\ref{EquCty2})$}.\\
This result comes from the structure of convex functions, and of their gradients, the monotone functions, so we first state a result that helps us apprehend the topography of a monotone function.
\begin{Lm}
\label{LmTopMon}
There is a family $\{K_\epsilon\,\varepsilon>0\}$ of closed subsets of $K$,  satisfying

\begin{description}
\item{$\mathbf{(i)}$} $\epsilon>\epsilon' \Rightarrow K_{\epsilon} \subset K_{\epsilon'}$,
\item{$\mathbf{(ii)}$} $\bigcup_{\epsilon>0}K_\epsilon=\inr(K)$,
\item{$\mathbf{(iii)}$} For any convex function $f$, positive number $A$ and $\epsilon>0$,
\begin{eqnarray*}
\mu_{\| \nabla f\| }([A,\infty[)\leq \epsilon \Rightarrow \A z \in K_\epsilon, \| \nabla f(z) \| \leq 2A.
\end{eqnarray*}
\end{description}
\end{Lm}
Hence one can control the locations of points where $f$'s gradient reaches high values. In particular, $\| \nabla f \|$ cannot be ``too large'' far from the edges of $K$.
\begin{proof}

Any convex function $f$  on $K$ satisfies
\begin{eqnarray*}
\A ~z,\zeta \in K, \langle \nabla f(z)-\nabla f(\zeta),z-\zeta  \rangle \geq 0.
\end{eqnarray*}
It readily follows from the fact that the restriction of $f$ to $[z,\zeta]$ is convex.
Now, for $z \in K, u \in \mathbb{R}^d$, we introduce the affine cone

\begin{eqnarray*}
Z(z,u)=\{y \in K_d:~\langle y-z,u \rangle \geq \frac{1}{2}\| z-y \| \| u \| \}.
\end{eqnarray*}
We have the property that 
\begin{eqnarray*}
y \in Z(z, \nabla f(z)) \Rightarrow \| \nabla f(y) \| \geq \frac{1}{2} \| \nabla f(z) \|.
\end{eqnarray*}

Indeed, let $y$ be in $Z(z,\nabla f(z))$.
\begin{eqnarray*}
\| \nabla f(y) \| \| y-z \| \geq \langle  \nabla f(y),y-z \rangle \geq \langle  \nabla f(z),y-z \rangle \geq \frac{1}{2} \| \nabla f(z) \|  \| y-z \| .\\
\end{eqnarray*}

It means that  $y$ in the cone $Z(z,\nabla f(z))$ cannot have a gradient too small, due to the monotonicity property.
Now we set $\epsilon(z)=\inf_{u \in \mathcal{S}^{d-1}} \lambda_d(Z(z,u))$, which simply plays the role of a lower bound for $\lambda_d(Z(z,\nabla f(z)))$.
We have, for $z\in K$,
\begin{eqnarray}
\label{machintruc}
\lambda_d(\{y \in K:~ \|\nabla f(y)  \| \geq \frac{1}{2} \| \nabla f(z)  \|\}) \geq \lambda_d(Z(z,\nabla f(z)) \geq \epsilon(z).
\end{eqnarray}
Now we set, for $\epsilon>0$, $K_\epsilon=\{z \in K:~\epsilon(z) \geq \epsilon\}$. For  $z$ in $K_{\epsilon}$, $$ \| \nabla f(z)  \| \geq 2A \Rightarrow  \mu_{ \|\nabla  f \| }([A,\infty[) \geq  \lambda_d(\{y \in K:~ \|\nabla f(y)  \| \geq \frac{1}{2} \| \nabla f(z)  \|\} \geq \epsilon(z).$$

Hence, given any positive number $A$, if $\nabla f$ satisfies

\begin{eqnarray*}
 \mu_{ \|\nabla f_n  \| }([A,\infty[)\leq \varepsilon
\end{eqnarray*}
for some $\varepsilon>0$, then,  according to (\ref{machintruc}), it follows that for $z\in K_\varepsilon$
\begin{eqnarray*}
  \epsilon(z) \geq \varepsilon, \textrm{ and so }\| \nabla f(z)  \| \leq 2A.
\end{eqnarray*}

\end{proof}
To finish the proof of the theorem, we have to show that $f_n$ converges to $f$ on $\mathop{\textrm{int}}(K)$. In a first time we will use Ascoli-Arzela theorem to show that the $f_n$ uniformly converge on every $K_\epsilon$, and by consistency they converge pointwise on $\mathop{\textrm{int}}(K)$. Then we will show that the limit can be nothing but $f$.\\

Since $\mu_{\nabla f_n}$ weakly converges to the finite measure $\mu_{\nabla f}$, it is a tight family of measures. For all $\epsilon>0$, we can find $A>0$ such that, for all $n$ in $\mathbb{N}$,
\begin{eqnarray*}
 \mu_{ \|\nabla  f_n\| }([A,\infty[)\leq \epsilon.
\end{eqnarray*}
Hence, according to Lemma \ref{LmTopMon},
\begin{eqnarray*}
\A n \in \mathbb{N}, \A z \in K_\epsilon, \| \nabla f_n(z) \| \leq 2A.
\end{eqnarray*}
For a function $u$, call $u^\varepsilon$ its restriction to $K_{\varepsilon}$.
According to Ascoli-Arzela criterion, we know that for all $\epsilon>0$, $\{f_n^\varepsilon;~ n \geq 1\}$ is a relatively compact family for the uniform convergence. Now, let $\epsilon$ be a positive number. There exists a convex function $f_\epsilon$ and a sub-sequence $f_{\varphi^\epsilon(n)}$ such that $f_{\varphi^\epsilon(n)}\to f_\epsilon$ uniformly on $K_{\epsilon}$.
 Let us show that $f_{\epsilon}$ coincides with $f$, which means that $f$ is in fact the limit as only possible limit for a sub-sequence.

By taking iteratively subsequences with the same arguments, one can complete $f_{\varepsilon}$ to a function $\tilde{f}$ on all $\inr(K)$ such that, for each $k\geq 1$, $f_{\varphi_{k}(n)}(z)\to \tilde{f}(z)$ for $z$ in $K_{\varepsilon/k}$, where $\varphi_{1}=\varphi^\varepsilon$, and $\varphi_{k}(n)$ is a subsequence of $\varphi_{k-1}(n)$.

In particular, using a diagonal extraction, there is a subsequence   $f_{\phi(n)}$ that converges pointwise to $\tilde{f}$ on $K_{\varepsilon}$. According to the result $(\ref{EquCty1}) \Rightarrow (\ref{eq:CvgMsr})$ proved earlier, we know that $\mu_{\nabla f_{n}} \Rightarrow \mu_{\nabla \tilde{f}}$, and so, by unicity of the limit, $\mu_{\nabla \tilde{f}}=\mu_{\nabla f}$.

Hence $\nabla \tilde{f}$ and $\nabla f$ are two monotone functions on $K$ whose distributions coincide. The uniqueness in Brenier's theorem (Th. \ref{ThmBre}) ensures us that they are equal a.e..
We have proved that any cluster point $f^\epsilon$ of $(f_n(z), z \in K_\epsilon)$ is equal to $f$ on $K_{\epsilon}$. Hence $f$ is the limit of $f_n$ for the uniform convergence on $K_\epsilon$. Since for convex functions on a convex compact set, uniform convergence and pointwise convergence are equivalent, we have
\begin{eqnarray*}
\A \epsilon>0, \| f_n(z)-f(z) \|_{\infty}^{K_{\epsilon}}\to 0
\end{eqnarray*}
which yields the result.

\subsection{Proof of Theorem \ref{CorCvgGss}}
\label{PrfThmCvgMunGss}

For the sake of clarity, we drop the exponent ``$\mathcal{T}$'' in the proof, so $\varphi_{n}^\mathcal{T}=\varphi_{n}$.
We consider here the quantity
\begin{equation*}
\varphi_{n}(h)=\int_{K_{d}}\exp(\imath \langle h,Y_{n}(z)\rangle) dz.
\end{equation*}
This integral can be discretized in a sum over all simplices  $S$ of $\mathcal{T}_{n}$. For $S$ in $\mathcal{T}_{n}$, denote by $Y_{n}(S)$ (resp. $\Lambda_{n}(S)$) the common value of $Y_{n}$ (resp. $\Lambda_{n}$)  over $S$. We have
\begin{equation*}
\varphi_{n}(h)=\sum_{S\in\mathcal{T}_{n}}\exp(\imath \langle h, Y_{n}(S)\rangle)\vol(S).
\end{equation*}

To prove that $\varphi_{n}(h)-\mathbb{E}\varphi_{n}(h)$ converges a.s. to $0$, we study the summability of the 4-th order moment
\begin{align*}
\mathbb{E}|\varphi_{n}(h)&-\mathbb{E}\varphi_{n}(h)|^4\\
&=\mathbb{E}\left[(\varphi_{n}(h)-\mathbb{E}\varphi_{n}(h))^2\overline{\varphi_{n}(h)-\mathbb{E}\varphi_{n}(h)}^2\right]\\
&=\mathbb{E}\prod_{k=1}^4 \left[\sum_{S_{k}\in\mathcal{T}_{n}} \vol(S_{k})(\exp(\imath \langle h ,\varepsilon_{k}Y_{n}(S_{k})\rangle)-\mathbb{E}\exp(\imath \langle h ,\varepsilon_{k}Y_{n}(S_{k})\rangle))\right]
\end{align*}
with $\varepsilon_{1}=\varepsilon_{2}=-\varepsilon_{3}=-\varepsilon_{4}=1$. Since $\varepsilon_{k}Y_{n}(S_{k})$ is a Gaussian vector with covariance matrix $\Lambda_{n}(S_{k})$, we have
\begin{align*}
\mathbb{E}|\varphi_{n}(h)&-\mathbb{E}\varphi_{n}(h)|^4\\
&=\mathbb{E}\prod_{k=1}^4 \left[\sum_{S_{k}\in\mathcal{T}_{n}} \vol(S_{k})(\exp(\imath \langle h ,\varepsilon_{k}Y_{n}(S_{k})\rangle)-\exp(-1/2 \langle h,\Lambda_{n}(S_{k}) h\rangle))\right].
\end{align*}
If one develops the previous quantity, one obtains the sum of all products of four terms of the form  $\exp(\imath \langle h ,\varepsilon_{k}Y_{n}(S_{k})\rangle)$ or $-\exp(-1/2 \langle h,\Lambda_{n}(S) h\rangle)$.

 Denote by $\mathcal{P}$ the class of all subsets of $\{1,2,3,4\}$. Summing over all possible quadruples $Q=(S_{1},S_{2},S_{3},S_{4})$, and all possibles ways to write four terms of one of the two forms described above, one obtains, with $\vol(Q)=\prod_{k=1}^4\vol(S_{k})$,
\begin{align}
\label{eq:development-4thmoment}
\nonumber\mathbb{E}&|\varphi_{n}(h)-\mathbb{E}\varphi_{n}(h)|^4\\
\nonumber&=\sum_{Q \in \mathcal{T}_{n}^4}\vol(Q) \sum_{P\in\mathcal{P}_{4}}\mathbb{E}\prod_{k\in P}\exp(\imath \langle h, \varepsilon_{k}Y_{n}(S_{k})\rangle)\prod_{k\notin P} (-\exp(-\langle h,\Lambda_{n}(S_{k}) h\rangle))\\
&=\sum_{Q \in \mathcal{T}_{n}^4}\vol(Q) \sum_{P\in\mathcal{P}_{4}}(-1)^{|4-P|}\mathbb{E}\prod_{k\in P}\exp(\imath \langle h, \varepsilon_{k}Y_{n}(S_{k})\rangle)\prod_{k\notin P} \exp(-\langle h,\Lambda_{n}(S_{k}) h\rangle).
\end{align}
Since $\sum_{k\in P}\varepsilon_{k}Y_{n}(S_{k})$ is a Gaussian vector, one gets  
\begin{equation*}
\mathbb{E}\prod_{k\in P}\exp(\imath \langle h, \varepsilon_{k}Y_{n}(S_{k})\rangle)=\exp\left(-1/2\left\langle h, \cov\left( \sum_{k\in P}\varepsilon_{k}Y_{n}(S_{k})\right)h\right\rangle\right).
\end{equation*}
The point of this computation is that $\cov(\sum_{k}\varepsilon_{k}Y_{n}(S_{k}))$ should be close to $\sum_{k\in P}\Lambda_{n}(S_{k})$. Indeed, if simplices $S_{1},S_{2},S_{3},S_{4}$ are far from each other, the corresponding random variables $Y_{n}(S_{k}), k=1,\dots,4$ have small dependancy, provided $\sigma$ is regular enough. Thus we introduce the matrix
\begin{align*}
\chi_{n}^P(Q)&=\cov\left(\sum_{k\in P}\varepsilon_{k}Y_{n}(S_{k})\right)-\sum_{k\in P}\cov(Y_{n}(S_{k}))\\
&=\sum_{k\neq k'\in P}\varepsilon_{k}\varepsilon_{k'}\cov(Y_{n}(S_{k}),Y_{n}(S_{k'})).
\end{align*}
We can decompose the summand in (\ref{eq:development-4thmoment}) in 
\begin{equation*}
\mathbb{E}\prod_{k\in P}\exp(\imath \langle h, \varepsilon_{k}Y_{n}(S_{k})\rangle)\prod_{k\notin P} \exp(-\langle h,\Lambda_{n}(S_{k}) h\rangle)=\psi_{n}(Q)\exp(-1/2\langle h,\chi_{n}^P(Q) h \rangle)
\end{equation*}
where
\begin{equation*}
\psi_{n}(Q)=\prod_{k=1}^4 \exp(-1/2\langle h,\Lambda_{n}(S_{k}) h\rangle)=\exp\left(-1/2\left\langle h,\sum_{k=1}^4 \Lambda_{n}(S_{k}) h \right\rangle \right)
\end{equation*}
does not depend on $P$. If we develop the exponential at the 2d order, we have
\begin{align}
\label{eq:dev-exponential}
\nonumber\exp & \left(-1/2 \langle  h,\chi_{n}^P(Q)h \rangle \right)=\\
&1-\frac{1}{2} \langle h,\chi_{n}^P(Q)h \rangle+\frac{1}{8}\exp(-\theta/2   \langle h,\chi_{n}^P(Q)h \rangle ) (\langle h,\chi_{n}^P(Q)h \rangle)^2
\end{align}
for some $\theta$ in $[0,1]$.

For $0 \leq c \leq 4$, let $\mathcal{P}_{c}$ be the class of elements of $\mathcal{P}$ that have cardinality $c$. Remark that
 \begin{eqnarray*}
\sum_{P \in \mathcal{P}} (-1)^{|P|}&=&\sum_{c=0}^4\sum_{P\in \mathcal{P}_c} (-1)^{c}\\
&=&1-4+6-4+1=0.
\end{eqnarray*}
In view of computing the first order term in (\ref{eq:development-4thmoment}), we have for $1\leq i,j \leq d$
\begin{align*}
\sum_{Q\in \mathcal{T}_{n}^4}&\vol(Q) \psi_n(Q) \sum_{P \in \mathcal{P}} (-1)^{|P|} \chi_{n}^P(Q)_{i,j} \\
&=\sum_{P \in \mathcal{P}}(-1)^{|P|}\sum_{k,k' \in \mathcal{P}:\,k \neq k'}\epsilon_k\epsilon_{k'} \sum_{Q \in \mathcal{T}_n^4}\vol(Q)\cov(Y_{n}(S_{k}),Y_{n}(S_{k'}))\psi_n(Q)\\
&=\sum_{1\leq k,k'\leq 4\atop \,k\neq k' }\epsilon_k\epsilon_{k'} \sum_{Q\in \mathcal{T}_n^4} \vol(Q) \cov(Y_{n}(S_{k}),Y_{n}(S_{k'}))\psi_n(Q)\sum_{P \in \mathcal{P}:\,k,k'\in P}(-1)^{|P|}.
 \end{align*}
Take $k \neq k' $ in $\{1,2,3,4\}$. There are exactly one $P$ of $\mathcal{P}_2$, 2 sets $P$ in $\mathcal{P}_3$ and $1$ set of $\mathcal{P}_4$ that contain $q$ and $q'$. Hence
\begin{eqnarray*}
\sum_{{P \in \mathcal{P}}\atop {P \ni kk'}}(-1)^{|P|}=1-2+1=0.
\end{eqnarray*}
Thus, when we inject the development (\ref{eq:dev-exponential}) in the sum (\ref{eq:development-4thmoment}), the main and first order terms vanish, and only the second order term remains,
\begin{align*}
\mathbb{E}|\varphi_{n}(h)&-\mathbb{E}\varphi_{n}(h)|^4\\
&=\frac{1}{8}\sum_{Q \in \mathcal{T}_{n}^4} \vol(Q)\psi_{n}(Q)\sum_{P\in \mathcal{P}}(-1)^{|P|}\exp(-\theta/2 \langle h,\chi_{n}^P(Q) h \rangle) \langle h, \chi_{n}^P(Q) h\rangle^2\\
&=O\left( \sum_{Q \in \mathcal{T}_{n}^4} \vol(Q)\sum_{P\in \mathcal{P}}\psi_{n}(Q)\exp(-\theta/2 \langle h,\chi_{n}^P(Q) h \rangle) \langle h, \chi_{n}^P(Q) h\rangle^2\right).
\end{align*}
Since $\psi_{n}(Q)$ is a product of characteristic functions, it is smaller than $1$.
If for some $Q,P,$ $\langle h,\chi_{n}^P(Q) h \rangle$ is positive, then the term  $\psi_{n}(Q)\exp(-\theta/2 \langle h,\chi_{n}^P(Q) h \rangle)$ is smaller than $1$.
If on the contrary it is negative, then $-\theta/2 \langle h,\chi_{n}^P(Q) h \rangle \leq -1/2  \langle h,\chi_{n}^P(Q) h \rangle $, and
\begin{align*}
&\psi_{n}(Q)\exp(-\theta/2 \langle h,\chi_{n}^P(Q) h \rangle)\leq \psi_{n}(Q)\exp(-1/2 \langle h,\chi_{n}^P(Q) h \rangle)\\
&= \exp\left(-1/2\left\langle h,\sum_{k\notin P}\cov(S_{k})h \right\rangle\right) \exp\left(-1/2 \left\langle h, \left(\sum_{k\in P}\cov(Y_{n}(S_{k}))+ \chi_{n}^P(Q)\right) h  \right\rangle\right)\\
& = \exp\left(-1/2 \left\langle h,\sum_{k\notin P}\cov(S_{k})h \right\rangle\right)  \exp\left(-1/2 \left\langle h, \cov\left(\sum_{k\in P}\varepsilon_{k}Y_{n}(S_{k})\right) h  \right\rangle\right).\\
\end{align*}
This is again a product of characteristic functions, hence smaller than $1$, and we have
\begin{align*}
\mathbb{E}|\varphi_{n}(h)&-\mathbb{E}\varphi_{n}(h)|^4 =O\left( \sum_{Q \in \mathcal{T}_{n}^4;P\in \mathcal{P}}\vol(Q) \langle h,\chi_{n}^P(Q)h\rangle^2\right).
\end{align*}
By writing explicitly $\chi_{n}^P(Q)$, we arrive at 
\begin{align}
\label{eq:final-bound-4thmoment}
\nonumber\mathbb{E}&\varphi_{n}(h)-\mathbb{E}\varphi_{n}(h)|^4 \\
\nonumber&=O\left(\sum_{Q\in\mathcal{T}_{n}^4}\vol(Q)|\mathbb{E}Y_{n,u}(S_{1})Y_{n,v}(S_{2})
||\mathbb{E}Y_{n,u}(S_{3})Y_{n,v}(S_{4}) |\right)\\
&=O\left(\sum_{S,S'\in\mathcal{T}_{n}}\vol(S)\vol(S')|\mathbb{E}Y_{n,u}(z)Y_{n,v}(\zeta)|\right)^2
\end{align}
where $[z,z+n^{-1}u]$ and $[\zeta,\zeta+n^{-1}v]$ are  edges of $S$ and $S'$, respectively. Applying (\ref{eq: cov-gradient}) yields the result.

\section*{Acknowledgements}
The authors wish to warmly thank Pr.Ilya Molchanov, who suggested the use of triangulations, and more generally for enriching discussions that contributed  to bring this article to its maturity. We also are very grateful to the referee, who, by its careful reading and judicious remarks, helped improving this article.

\end{document}